\documentclass{siamltex}
\usepackage{amssymb}
\usepackage{bm} 
\usepackage{graphicx}
\usepackage{graphics}
\usepackage{caption2}
\usepackage{psfrag}
\usepackage{amsmath}
\usepackage{amscd}
\usepackage{amsfonts}
\usepackage{float}
\usepackage{latexsym}
\usepackage{lscape}
\usepackage{color}
\usepackage{soul}

\newtheorem{remark}{Remark}[section]
\newtheorem{assumption}{Assumption}[section]

\newcommand{\appvt}[1]{{{}\hat{v}_{#1}^{(n)}}^\top}
\newcommand{\appv}[1]{\hat{v}_{#1}^{(n)}}

\newcommand{\vv}[1]{ v_{#1}^{(n)} }
\newcommand{\ds}{\displaystyle}
\newcommand{\vt}[1]{{v_{#1}^{(n)}}^\top}

\usepackage[noend]{algpseudocode}
\usepackage{algorithmicx,algorithm}
\begin{document}
\title{Convergence analysis of discrete high-index saddle dynamics}
\author{
Yue Luo\thanks{Beijing International Center for Mathematical Research, Peking University, Beijing 100871, China (moonluo@pku.edu.cn)}
\and
Xiangcheng Zheng\thanks{School of Mathematical Sciences, Peking University, Beijing 100871, China (zhengxch@math.pku.edu.cn)}
\and
Xiangle Cheng\thanks{Huawei 2012 Network Technology Lab, Beijing 100095, China (chengxiangle1@huawei.com)}
\and
Lei Zhang\thanks{Beijing International Center for Mathematical Research, Center for Machine Learning Research, Center for Quantitative Biology, Peking University, Beijing 100871, China (zhangl@math.pku.edu.cn)}
} 
\maketitle

\begin{abstract}
Saddle dynamics is a time continuous dynamics to efficiently compute the any-index saddle points and construct the solution landscape. In practice, the saddle dynamics needs to be discretized for numerical computations, while the corresponding numerical analysis are rarely studied in the literature, especially for the high-index cases. In this paper we propose the convergence analysis of discrete high-index saddle dynamics. To be specific, we prove the local linear convergence rates of numerical schemes of high-index saddle dynamics, which indicates that the local curvature in the neighborhood of the saddle point and the accuracy of computing the eigenfunctions are main factors that affect the convergence of discrete saddle dynamics. The proved results serve as compensations for the convergence analysis of high-index saddle dynamics and are substantiated by numerical experiments.
\end{abstract}

\begin{keywords}
saddle dynamics, solution landscape, saddle point, convergence 
\end{keywords}

\begin{AMS}
37M05, 37N30, 65L20
\end{AMS}

\pagestyle{myheadings}
\thispagestyle{plain}

\markboth{Luo, Zheng, Cheng and Zhang
}{Convergence analysis of discrete high-index saddle dynamics}

\section{Introduction}
Locating the saddle points in complex systems has been of broad interest in many fields of scientific research. A large spectrum of examples includes finding the critical nuclei and transition pathways in phase transformations \cite{cheng2010nucleation, Han2019transition, samanta2014microscopic, wang2010phase, Yin2020nucleation, zhang2007morphology}, the defect configurations in liquid crystals\cite{Han2021,han2021solution,wang2021modeling, YinPRL,yin2022solution}, the transition rates in chemical reactions and biology \cite{baker1986algorithm, wales2003energy}. Meanwhile, various numerical algorithms have been proposed to compute the saddle points and applied to different practical problems, such as the gentlest ascent dynamics \cite{gao2015iterative,gad}, the dimer-type methods \cite{gould2016dimer, henkelman1999dimer, zhangdu2012, zhang2016optimization}, the minimax method \cite{li2001minimax}, the activation-relaxation technique \cite{cances2009some}, {\it etc}. We refer to \cite{vanden2010transition, henkelman2002methods, zhang2016recent} as some excellent reviews.

Recently, the high-index saddle dynamics has been proposed to serve as a powerful instrument in finding the any-index saddle points \cite{yin2019high}. It plays a key role of construction of the solution landscapes \cite{YinPRL,YinSCM}. Here $x^*$ is called a non-degenerate index$-k$ ($1\leq\ k \in\mathbb{N}$) saddle point of $E(x)$ if the gradient $\nabla E(x^*)=0$ and the Hessian $\nabla^2 E(x^*)$ has exactly $k$ negative eigenvalues with no zero eigenvalue. The saddle dynamics for an index$-k$ saddle point reads
\begin{equation}
    \left\{
    \begin{aligned}
    \frac{dx}{dt} & =- \beta\bigg(I-2\sum_{i=1}^kv_iv_i^\top\bigg)\nabla E(x),\\
    \frac{dv_i}{dt} &=- \gamma\bigg(I-v_iv_i^\top-2\sum_{j=1}^{i-1}v_jv_j^\top\bigg)\nabla^2E(x)v_i,\ 1\leq i \leq k,
    \end{aligned}
    \right.
    \label{saddle_dyna}
\end{equation}
where $\beta$, $\gamma$ are positive relaxation parameters. This dynamical system is derived by the formulation of the minimax optimization for an index$-k$ saddle point and the construction of the maximal subspace. It was shown in \cite{gad,YinSCM,yin2019high} that a linear stable steady state of (\ref{saddle_dyna}) is an index$-k$ saddle point of $E$. Therefore, investigating the convergence of the iterations in discrete saddle dynamics to its limit is of great importance in practical computations.

{There has been some progresses} on the convergence analysis of the index$-1$ saddle points. In \cite{zhangdu2012}, the shrinking dimer dynamics is proposed to find the index-1 saddles, and both linear local asymptotic stability analysis and optimal convergence rates are presented. \cite{gould2016dimer} proves the local linear convergence rate for a dimer-type saddle search algorithm with preconditioning and line search. \cite{ortner2017} extends the local convergence rate analysis for dimer and gentlest ascent saddle search algorithms to the estimation on the region of attraction of saddles. \cite{gao2015iterative} provides a saddle point search algorithm under the iterative minimization formulation and proves the local convergence rate on condition that each subproblem is solved exactly.

{Despite the growing numerical analysis for the index$-1$ saddle points, the analytical studies for the index$-k$ saddle dynamics are far from well developed. Since the index$-k$ ($k>1$) saddle point has more unstable directions than the index$-1$ saddle point, the index$-1$ solvers are limited to locate only index$-1$ saddle points and the high-index saddle dynamics is required to find the index$-k$ saddle points. Compared with the analysis of the index$-1$ saddle dynamics with only one unstable eigenvector, the high-index saddle dynamics needs to analyze the multi-dimensional unstable subspace spanned by $k$ eigenvectors, i.e., $v_1,...,v_k$. Due to the strong nonlinearity and the orthonormalization procedure, much more technical matrix decompositions and analysis are required to obtain the desired results.}
 Motivated by these discussions, the main contribution of this work lies in providing the convergence analysis of the discrete high-index saddle dynamics. We theoretically show the local linear convergence rates of the numerical schemes, which rely on the local curvature in the neighborhood of the target index$-k$ saddle point and the accuracy of computing the eigenvectors. 
The developed analysis and results offer a perspective for analyzing the gradient-based saddle point searching algorithms and provide mathematical supports for the convergence rates when implementing the saddle dynamics in applications.

The rest of the paper is organized as follows: In Section 2 we introduce the notations, assumptions and commonly-used algorithms for high-index saddle dynamics. In Section 3 we prove the auxiliary lemmas to be used in the convergence analysis. In Section 4 we prove the convergence results for the index-1 saddle dynamics under both exact and approximated eigenvectors. We then extend the developed results for index-$k$ saddle dynamics in Section 5. In Section 6 we present some numerical experiments to substantiate the theoretical findings. We finally address concluding remarks in Section 7.

\section{Preliminaries}

\subsection{Notations and assumptions} 
Let $E:\mathbb{R}^d\rightarrow \mathbb{R}$ be a real valued, twice differentiable energy function defined on $d-$dimensional Euclidean space and $x^*$ be a non-degenerate index-$k$ saddle point. Define $\|\cdot\|_2:\mathbb{R}^{d\times d}\rightarrow \mathbb{R}$ as the operator norm of $d\times d$ real matrices 
\[\|A\|_2=\max_{\|x\|_2=1}\|Ax\|_2,~~x\in\mathbb{R}^d,~~\|x\|^2_2: = \sum_{i=1}^d x_i^2.\] 
 Let $x^{(n)}\in\mathbb{R}^d$ be the position variable of the $n^{th}$ iteration. Denote eigen-pairs of $\nabla^2 E(x^{(n)})$ as  $\big\{(\lambda_i^{(n)},v_i^{(n)})\big\}_{i=1}^d$, i.e., $\nabla^2 E(x^{(n)})v_i^{(n)}=\lambda_i^{(n)}v_i^{(n)}$, $i=1,2,...,d$.  Here the eigenvalues are sorted in the following order
$\lambda_1^{(n)}\leq \lambda_2^{(n)} \leq ...\leq \lambda_d^{(n)}$ and the eigenvectors satisfy ${{} v_i^{(n)}}^\top  v_j^{(n)}=\delta_{ij}$ for $1\leq i,j\leq d$. Furthermore, for symmetric matrices $A$ and $B$, we denote $B\preceq A$ or $A\succeq B$ if $A-B$ is a positive semidefinite matrix.

We make the following  assumptions  throughout our paper: 
\begin{assumption}
The initial position $x^{(0)}$ {\color{blue}is} in a neighborhood of $x^*$, i.e., $x^{(0)}\in U(x^*,\delta) =\{x|\|x-x^*\|_2<\delta\}$ for some $\delta>0$ such that
    \begin{enumerate}
    \item[(i)] There exists a constant $M>0$ such that $\|\nabla^2 E(x)-\nabla^2 E(y)\|_2\leq M\|x-y\|_2$ for $ x,y\in U(x^*,\delta)$;
   \item[(ii)] For any $x\in U(x^*,\delta)$, eigenvalues $\{\lambda_i\}_{i=1}^d$ of $\nabla^2E(x)$ satisfy $\lambda_1\leq\cdots\leq \lambda_k<0< \lambda_{k+1}\leq\cdots\leq\lambda_d$ 
         and there exist positive constants $0<\mu< L$ such that $|\lambda_i|\in[\mu,L]$ for $1\leq i\leq d$.
\end{enumerate}
\label{original_asm}
\end{assumption}
\begin{remark}
    The assumption (i) holds for all smooth $E(x)$, while the assumption (ii) is natural since the eigenvalues $\{\lambda_i^*\}_{i=1}^d$ of $\nabla^2 E(x^*)$ satisfy
$\lambda_1^*\leq \lambda_2^* \leq\cdots \leq\lambda_k^* <0<\lambda_{k+1}^*\leq \cdots\leq \lambda^*_{d}$ and $x\in U(x^*,\delta)$ could be close to $x^*$. {The assumption $x^{(0)}\in U(x^*,\delta)$ is standard to analyze the local convergence behavior of optimization algorithms \cite{nesterov2003introductory}. In practice, though $x^{(0)}$ may not be close enough to $x^*$ as assumed, the numerical solution $x^{(n)}$ at the $n$th step will approach $x^*$ and then come into some small neighborhood of $x^*$ after certain steps since the index$-k$ saddle points are attractors of high-index saddle dynamics. As what we interest for the convergence result is the convergence behavior as $n\rightarrow \infty$, the assumption $x^{(0)}\in U(x^*,\delta)$ is usually reasonable without loss of generality.}
\end{remark}
{
\begin{remark}
	The assumption (ii) on the upper and lower bounds of eigenvalues follows from the conventional treatments of optimization methods \cite{nesterov2003introductory} for the sake of numerical analysis. For some specific and commonly-used functions, we are able to estimate $\mu$ and $L$. For instance, if $E(x)$ is quadratic, $L$ and $\mu$ could be estimated by Lanzcos methods. For general functions, we could choose $L$ large enough and $\mu$ small enough to ensure the validity of the assumption (ii). 
\end{remark}

}

\subsection{Numerical implementation of saddle dynamics}\label{sec22}
When computing an index$-k$ saddle point, the saddle dynamics (\ref{saddle_dyna}) is often implemented under the framework of the following algorithm \cite{YinSCM,yin2019high}. 
\begin{algorithm}[H]
 \caption{Saddle dynamics (\ref{saddle_dyna}) for an index$-k$ saddle point}
 \hspace*{0.02in} {\bf Input:} 
 $k\in\mathbb{N}$, $x^{(0)}\in\mathbb{R}^d$, $\big\{\hat v_i^{(0)}\big\}_{i=1}^k\subset \mathbb{R}^d$ satisfying ${{}\hat v_i^{(0)}}^\top \hat v_j^{(0)}=\delta_{ij}$.\\
 \begin{algorithmic}[]
  \For{$n=0,1,...,T-1$}
     \State $\ds x^{(n+1)} = x^{(n)} - \beta_n\bigg(I-2\sum_{i=1}^k\hat v_i^{(n)}{{}\hat v_i^{(n)}}^\top\bigg)\nabla E(x^{(n)})$; \vspace{0.05in}
    \State $\big\{\hat v_i^{(n+1)}\big\}_{i=1}^k = \text{EigenSol}\big( \big\{\hat v_i^{(n)}\big\}_{i=1}^k, \nabla^2 E(x^{(n+1)}) \big)$.\vspace{0.1in}
  \EndFor
 \hspace{-0.2in} {\bf Return:} $x^{(T)}$
 \end{algorithmic}
 \label{alg_1}
\end{algorithm}
In Algorithm \ref{alg_1}, {the index $k$ is set as a priori. Consequently, the index-$k$ saddle dynamics is implemented to locate the desired index-$k$ saddle point.}
 EigenSol represents some specific eigenvector computation solver, with $\big\{\hat v_i^{(n)}\big\}_{i=1}^k$ as its initial values, for the computation of eigenvectors corresponding to $k-$smallest eigenvalues of $\nabla^2 E(x^{(n+1)})$. In \cite{yin2019high}, two numerical methods have been proposed. The first method is based on the explicit Euler discretization of the dynamics of $v_i$ in (\ref{saddle_dyna}) with the application of the dimer method for approximating the multiplication of the Hessian and the vector, i.e.,
\begin{equation*}
    \begin{aligned}
    & H(x^{(n+1)},\hat  v_i^{(n)}, l^{(n)}) = \frac{1}{2l^{(n)}}\left[ \nabla E (x^{(n+1)}+l^{(n)}\hat v_i^{(n)}) - \nabla E (x^{(n+1)}-l^{(n)}\hat v_i^{(n)}) \right],\\
    & \hat v_i^* =\hat  v_i^{(n)} - \gamma_n\left(I-\hat v_i^{(n)}{{}\hat v_i^{(n)}}^\top-2\sum_{j=1}^{i-1}\hat v_i^{(n)}{{}\hat v_i^{(n)}}^\top\right)H(x^{(n+1)}, \hat v_i^{(n)},l^{(n)}),\quad 1\leq i\leq k,\\
    & \big\{\hat v_i^{(n+1)}\big\}_{i=1}^k = \text{Orth}\big( \left\{\hat v_i^*\right\}_{i=1}^k\big).
    \end{aligned}
\end{equation*}
Here Orth refers to {the modified Gram-Schmidt orthonormalization for the numerical stability}. This procedure is equivalent to the one step simultaneous Rayleigh-quotient iterative minimization method \cite{longsine1980simultaneous}. The other approach is based on the Locally Optimal Block Preconditioned Conjugate Gradient (LOBPCG) method \cite{knyazev2001toward} and EigenSol is defined as one step LOBPCG iteration. Both methods generate orthonormal vectors, i.e., ${{}\hat v_i^{(n)}}^\top \hat v_j^{(n)}=\delta_{ij}$ for $1\leq i,j\leq d.$ 
\section{Auxiliary lemmas} We prove several auxiliary lemmas to support the subsequent analysis. The following crucial lemma is usually applied for analyzing the convergence of gradient descent method \cite{nesterov2003introductory}, which is based on the idea of contracting mappings.
\begin{lemma}
Let $\{r_n\}_{n\geq 0}$ be a non-negative series satisfying
\begin{equation*}
    r_{n+1}\leq (1-q)r_n + cr_n^2,\quad n\geq 0,~~q\in(0,1),~~c>0.
\end{equation*}

\begin{enumerate}
    \item[(a)]  If $\displaystyle r_{n}<\frac{q}{c}$ for some $n\geq 0$, then 
    $\displaystyle r_{n+1}<r_n<\frac{q}{c};$
    \item[(b)] If $\displaystyle r_{0}<\frac{q}{c}$, then 
    $\displaystyle r_{n+1}\leq \left(\frac{1}{1+q}\right)^{n+1}\frac{qr_0}{q-cr_0}$ for all $n\geq 0$.
\end{enumerate}
\label{useful_lemma}
\end{lemma}

\begin{proof}
 If $\ds r_n<\frac{q}{c}$, then $q-cr_n>0$ and consequently
     \begin{equation*}
     {	r_{n+1}\leq r_n(1-q+cr_n)<r_n(1-q+c\cdot \frac{q}{c}) = r_n,}
     \end{equation*}
    which proves (a). To prove (b), we first find from (a) that $\ds r_0<\frac{q}{c}$ implies $\ds r_{n+1}<r_{n}<\frac{q}{c}$ for all $n\geq 0$.
Note that 
\begin{equation*}
\begin{aligned}
     r_{n+1}\leq r_n(1-q+cr_n) = r_n\frac{1-(q-cr_n)^2}{1+(q-cr_n)}\leq \frac{r_n}{1+(q-cr_n)},
\end{aligned}
\end{equation*}
which is equivalent to
\begin{equation*}
    \frac{q}{r_{n+1}} - c \geq \frac{q(1+q-cr_n)}{r_n}-c = (1+q)\left(\frac{q}{r_n}-c\right).
\end{equation*}
As $\frac{q}{r_n}-c>0$ for all $n\geq 0$, we get
\begin{equation*}
    \frac{q}{r_{n+1}}-c \geq (1+q)^{n+1}\left(\frac{q}{r_{0}}-c\right),
\end{equation*}
which leads to 
\begin{equation*}
    r_{n+1}\leq  \frac{q}{(1+q)^{n+1}(q/r_0-c)+c} \leq \left(\frac{1}{1+q}\right)^{n+1}\frac{qr_0}{q-cr_0}.
\end{equation*}
Thus we complete the proof.
\end{proof}

To analyze the convergence of the numerical saddle dynamics, the relation between $x^{(n+1)}-x^*$ and $x^{(n)}-x^*$ plays a critical role and we illustrate the recursion relation in the following lemma.
\begin{lemma}
    For any iteration scheme in the form of 
    \[x^{(n+1)} = x^{(n)} - \beta_n A^{(n)}\nabla E(x^{(n)}), \]
    where $A^{(n)}\in\mathbb{R}^{d\times d}$, we have the following identity
    \[x^{(n+1)}-x^* = \left[Q^{(n)} + B^{(n)}\right] (x^{(n)}-x^*)\]
   where
    \begin{equation*}
        \begin{aligned}
             &Q^{(n)} = I-\beta_n A^{(n)}\nabla^2 E(x^{(n)}),\\
             &B^{(n)} = \beta_n A^{(n)}\left[ \nabla^2 E(x^{(n)}) - \int_0^1 \nabla^2 E(x^*+t(x^{(n)}-x^*))dt\right].\\
        \end{aligned}
    \end{equation*}
    \label{recursion}
    Furthermore, if $x^{(n)} \in U(x^*,\delta)$, then we have
    \begin{equation*}
        \|B^{(n)}\|_2\leq \frac{1}{2}\beta_nM\|A^{(n)}\|_2\|x^{(n)}-x^*\|_2.
    \end{equation*}
\end{lemma}
\begin{proof}
    A direct calculation yields
    \begin{equation*}
    \begin{aligned}
         x^{(n+1)}-x^* & = x^{(n)} - x^* -\beta_n A^{(n)}\nabla E(x^{(n)}) \\
                       & = x^{(n)} - x^* -\beta_n A^{(n)}\left[\nabla E(x^{(n)})-\nabla E(x^*)\right] \quad (\text{Since $\nabla E(x^*)=0$})\\
                       & = \left[ I -\beta_n A^{(n)}\int_0^1\nabla^2 E(x^*+t(x^{(n)}-x^*))dt\right](x^{(n)} - x^*)\\
                       & =\left[Q^{(n)} + B^{(n)}\right] (x^{(n)}-x^*),
    \end{aligned}
    \end{equation*}
    { where in the third equality we used the integral residue of the Taylor expansion.}
    If $x^{(n)} \in B(x^*,\delta)$, then 
by assumption (i)
    \begin{equation*}
    \begin{aligned}
         \|B^{(n)}\|_2 
         &\leq \beta_n\|A^{(n)}\|_2\int_0^1 \|\nabla^2E(x^{(n)})-\nabla^2E(x^*+t(x^{(n)}-x^*))\|_2dt\\
         &\leq \beta_n\|A^{(n)}\|_2M\|x^{(n)}-x^*\|_2\int_0^1 1-t\,dt\\
         &= \frac{1}{2}\beta_nM\|A^{(n)}\|_2\|x^{(n)}-x^*\|_2,
        \end{aligned}
    \end{equation*}
    which completes the proof.
\end{proof}
        
\begin{lemma}
\label{Q}
For
   $\ds Q = I-\beta\bigg(\sum_{i=1}^dz_iu_iu_i^\top\bigg)$ where $\ u_i^\top u_j=\delta_{ij}$, $\beta>0,\ L\geq z_i\geq\mu> 0,\ 1\leq i\leq d$, the following estimate holds 
\begin{equation*}
    \|Q\|_2\leq \max\left\{|1-\beta L|,|1-\beta\mu|\right\}.
\end{equation*}
In particular, $\ds\beta=\frac{2}{L+\mu}$ leads to $\ds
    \|Q\|_2\leq \frac{L-\mu}{L+\mu}.$
\end{lemma}    
\begin{proof}
Denote $N = \sum_{i=1}^dz_iu_iu_i^\top$ such that $z_i$ is the eigenvalue of $N$ and $u_i$ is the corresponding eigenvector. Then $1-\beta z_i$ are eigenvalues of $Q$ and 
\begin{equation*}
    1-\beta L\leq 1-\beta z_i \leq 1-\beta \mu,\ 1\leq i\leq d.
\end{equation*}
Thus $
    \|Q\|_2 =  \max_{1\leq i\leq d}|1-\beta z_i| \leq \max\{|1-\beta L|, |1-\beta\mu|\}$,
which completes the proof.
\end{proof}
    
\begin{lemma}\cite{golub1996matrix}
Let $W,Z\in\mathbb{R}^{n\times n}$, $W_1,Z_1\in\mathbb{R}^{n\times k}$ and $W_2,Z_2\in\mathbb{R}^{n\times (n-k)}$ such that
\begin{equation*}
    W = \left[W_1 , W_2\right],\quad Z = \left[Z_1, Z_2\right].
\end{equation*}
If $W$ and $Z$ are orthogonal matrices, i.e., $WW^T=ZZ^T=I_n$, then 
\begin{equation*}
    \|W_1W_1^T - Z_1Z_1^T\|_2 = \|W_1^TZ_2\|_2 = \|Z_1^TW_2\|_2.
\end{equation*}
\label{subspace}
\end{lemma}

\section{Convergence rates of index-1 saddle dynamics}\label{sec4}
We start with the case of finding a non-degenerate index-1 saddle point $x^*$ of $E(x)$. Both the exact eigenvector $v^{(n)}_1$ of the smallest eigenvalue $\lambda_1^{(n)}$ of $\nabla^2 E(x^{(n)})$ and its approximation $\hat{v}_1^{(n)}$ computed via the schemes in Section \ref{sec22} will be applied in each iteration.

\subsection{The case of exact eigenvector $v^{(n)}_1$}\label{sec41}
Based on the scheme of the position variable
\begin{equation}
    x^{(n+1)} =  x^{(n)}-\beta_n\left(I-2v_1^{(n)}{v_1^{(n)}}^\top\right)\nabla E(x^{(n)}),\quad \|v_1^{(n)}\|_2=1,
    \label{exact_position}
\end{equation}
we first present the following single-step analysis.
\begin{theorem}
Under Assumption \ref{original_asm}, if 
$r_n:=\|x^{(n)}-x^*\|_2<\delta$
 and $\ds\beta_n = \frac{2}{L+\mu}$ for some $n\geq 0$, the following estimate holds
\begin{equation}\label{mh}
    r_{n+1}\leq \left(1-\frac{2\mu}{L+\mu}\right)r_n + \frac{Mr_n^2}{L+\mu}.
\end{equation}
  \label{exact_single}
\end{theorem}

\begin{proof}
Based on the formulation (\ref{exact_position}), we choose $A^{(n)} = I-2v_1^{(n)}\vt{1}$ and apply Lemma \ref{recursion} to get
\begin{equation*}
    x^{(n+1)}-x^* = \left[Q^{(n)} + B^{(n)}\right] (x^{(n)}-x^*).
\end{equation*}
Applying triangular inequality leads to 
\begin{equation}
    r_{n+1}\leq \|Q^{(n)}\|_2r_n + \|B^{(n)}\|_2r_n.
    \label{exact_relation}
\end{equation}
By Lemma \ref{recursion} and $||A^{(n)}||_2=1$
\begin{equation}
    \|B^{(n)}\|_2\leq \frac{1}{2}M\beta_nr_n.
    \label{exact_1_b}
\end{equation}
Then {it remains} to estimate $||Q^{(n)}||_2$. By eigenvalue decomposition theorem
\begin{equation}\label{eigendecom}
    \nabla^2E(x^{(n)})  = \sum_{i=1}^d \lambda_i^{(n)}v_i^{(n)}\vt{i},
\end{equation}
we have
\begin{equation*}
    \begin{aligned}
          Q^{(n)} &= 
          I-\beta_n\left(I-2v_1^{(n)}\vt{1}\right)\left(\sum_{i=1}^d \lambda_i^{(n)}v_i^{(n)}\vt{i}\right)\\
          &=I-\beta_n\left(\sum_{i=2}^{d}\lambda_i^{(n)}v_i^{(n)}\vt{i}-\lambda_1^{(n)}v_1^{(n)}\vt{1}\right).
    \end{aligned}
\end{equation*}
Since $r_n<\delta$, $-\lambda_1^{(n)}$ and $\lambda_i^{(n)}$ for {$ 2\leq i\leq d$} belong to $[\mu,L]$, we apply Lemma \ref{Q} to get 
\begin{equation}
    \|Q^{(n)}\|_2\leq \frac{L-\mu}{L+\mu}. 
    \label{exact_1_q}
\end{equation}
Then we combine (\ref{exact_relation})--(\ref{exact_1_q}) to complete the proof.
\end{proof}

We then present the main theorem of this subsection.
\begin{theorem}\label{thmorg}
    Under Assumption \ref{original_asm}, if the initial point $x^{(0)}$ satisfies
    \[r_0=\|x^{(0)}-x^*\|_2<\min\{\delta,\hat{r}\},\ \hat{r}=\frac{2\mu}{M} \] 
   and $\displaystyle \beta_n = \frac{2}{L+\mu} $ for any $n\geq 0$, $x^{(n)}$ converges to $x^*$ as $n\rightarrow \infty$ 
    with the estimate on the convergence rate
    \begin{equation}
        r_n=\|x^{(n)}-x^*\|_2\leq \left(1-\frac{2}{\kappa+3}\right)^n\frac{\hat{r}r_0}{\hat{r}-r_0},~~\kappa=\frac{L}{\mu}.
        \label{exact_1_conv}
    \end{equation}
\end{theorem}

\begin{proof}
By $r_0=\|x^{(0)}-x^*\|_2<\min\{\delta,\hat{r}\}$, we apply Theorem \ref{exact_single} with $n=0$ and Lemma \ref{useful_lemma} (a) with $q=\frac{2\mu}{L+\mu}\in(0,1]$ and $c=\frac{M}{L+\mu}>0$ to obtain
\begin{equation*}
    r_{1}<r_0<\min\{\delta,\hat{r}\},
\end{equation*}
which in turn implies that (\ref{mh}) holds for $n=1$.
Inductively, we could show that (\ref{mh}) holds for any $n\geq 0$. Then we apply Lemma \ref{useful_lemma} (b) to complete the proof.
\end{proof}

\subsection{The case of approximate eigenvector $\hat v^{(n)}_1$}
We generalize the convergence result in Section \ref{sec41} to {the case where the eigenvector} in each iteration is not exact and is computed from the numerical schemes in Section \ref{sec22}. In this case, the iteration scheme for the position variable reads
\begin{equation}
    x^{(n+1)} =  x^{(n)}-\beta_n\left(I-2\appv{1}\appvt{1}\right)\nabla E(x^{(n)}).
    \label{inexact_formulation}
\end{equation}
It is clear that there exists an $\alpha\in[0,1]$ such that
\begin{equation}
    1\geq |\vt{1}\hat{v}^{(n)}_1|^2 \geq 1-\alpha.
    \label{cos}
\end{equation}
In the following discussion, we assume (\ref{cos}) holds in each iteration.
\begin{remark}
Direct calculations show that
    \begin{equation*}
    |1 - \vt{1}\hat{v}^{(n)}_1 |  = |\vt{1} (\hat{v}_1^{(n)} - v_1^{(n)} )|\leq  \|\appv{1} - v_1^{(n)}\|_2.
    \end{equation*}
 It was proved in \cite{zhang2021optimalorder} that 
$\|\hat{v}_1^{(n)} - v_1^{(n)}\|_2 $
   has the first-order accuracy with respect to the time step size. Thus $\alpha$ could be arbitrarily close to $0$ by adjusting the time step size.
\end{remark}

\begin{lemma}
For $\ds
    D = \sum_{(i,j)\in S}\lambda_jc_ic_jv_iv_j^\top$ where $\ v_i^\top v_j=\delta_{ij}$, $\max_{1\leq j\leq d}|\lambda_j|\leq L$
and {$S=\{(i,j)|1\leq i,j\leq d\}\backslash\{(1,1)\}$}, the following estimate holds
\begin{equation*}
    \|D\|_2\leq L\|C_{-1}\|_2^2 + 2L|c_{1}|\|C_{-1}\|_2,~~C_{-1}:=\left[c_2,...,c_d\right]^\top.
\end{equation*}
\label{inexact_d}
\end{lemma}

\begin{proof}
We decompose $D$ as 
\begin{equation}
\begin{aligned}
     D &= \sum_{i=2}^d\sum_{j=2}^d \lambda_jc_{j}c_{i} v_i {v_j}^\top+\sum_{j=2}^d\lambda_jc_{j}c_{1}v_1{v_j}^\top+\sum_{i=2}^d\lambda_1c_{1}c_{i}v_i{v_1}^\top,
\end{aligned}
\end{equation}
and express the right-hand side terms as
\begin{equation*}
    \begin{aligned}
        \sum_{i=2}^d\sum_{j=2}^d\lambda_j c_{j} c_{i} v_i {v_j}^\top & = \bigg(\sum_{i=2}^d c_{i}v_i \bigg)\bigg(\sum_{j=2}^d \lambda_jc_{j}v_j\bigg)^\top= V_{-1}C_{-1}{{}C_{-1}}^\top\Lambda_{-1}
        {{} V_{-1}}^\top,\\
     \sum_{j=2}^d\lambda_jc_{j}c_{1}v_1{v_j}^\top 
     & = c_{1}v_1\bigg(\sum_{j=2}^d \lambda_jc_{j}v_j \bigg)^\top=c_{1}v_1{{}C_{-1}}^\top\Lambda_{-1}{{} V_{-1}}^\top,
\\
         \sum_{i=2}^d\lambda_1c_{1}c_{i}v_i{v_1}^\top 
         &= \bigg(\sum_{i=2}^d c_{i}v_i\bigg)\lambda_1c_{1} {v_1}^\top = \lambda_1 c_{1}V_{-1}C_{-1}{v_1}^\top,
    \end{aligned}
\end{equation*}
where $\Lambda_{-1} = \text{diag} \{\lambda_2,...,\lambda_d\}$ is a diagonal matrix and $V_{-1} = [v_2,...,v_d]\in\mathbb{R}^{d\times(d-1)}$ is a column-orthogonal matrix. We base on these to bound $\|D\|_2$ as 
\begin{equation*}
    \begin{aligned}
        \|D\|_2 
        & \leq \|C_{-1}{{}C_{-1}}^\top\Lambda_{-1}\|_2 + L|c_1|\|C_{-1}v_1^\top\|_2 + |c_1| \|v_1C_{-1}^\top \Lambda_{-1}\|_2\\
        & \leq L\|C_{-1}\|_2^2 + 2L|c_1|\|C_{-1}\|_2
    \end{aligned}
\end{equation*}
where we used $\|V_{-1}\|_2=1$ and  $\|\Lambda_{-1}\|_2=\max_{2\leq i\leq d}|\lambda_i|\leq L$. Thus we complete the proof.
\end{proof}

\begin{theorem}
Under assumption \ref{original_asm} and (\ref{cos}), 
    if $\alpha$ is small such that
    \begin{equation}
        0<\alpha<\frac{1}{2} \quad\text{and}\quad 1-2\alpha > 2\kappa(\alpha+2\sqrt{\alpha}),
    \end{equation}
where $\ds\kappa = \frac{L}{\mu}$, and $
    r_n=\|x^{(n)}-x^*\|_2<\delta
$ and $\ds\beta_n=\frac{2}{L+(1-2\alpha)\mu}$ for some $n\geq 0$,
then the following estimate holds
\begin{equation*}
    r_{n+1}\leq (1-q(\alpha))r_{n} + c(\alpha)r_n^2
\end{equation*}
where $\eta=1-2\alpha-2\kappa(\alpha+2\sqrt{\alpha})>0$ and
\begin{equation*}
     q(\alpha) = \frac{2\eta}{\kappa+(1-2\alpha)}\in(0,1), \quad 
    c(\alpha) = \frac{M\mu^{-1}}{\kappa+(1-2\alpha)}>0.
\end{equation*}

\label{inexact_single_step}
\end{theorem}

\begin{proof}
Based on (\ref{inexact_formulation}), let $A^{(n)} = I-2\appv{1}\appvt{1}$ and we apply Lemma \ref{recursion} to get
\begin{equation*}
    x^{(n+1)}-x^* = (Q^{(n)} + B^{(n)}) (x^{(n)}-x^*),~~\|B^{(n)}\|_2\leq\frac{1}{2}M\beta_n r_n.
\end{equation*}


We remain to bound $\|Q^{(n)}\|_2$. $\hat{v}_1^{(n)}$ could be represented by the basis $\{v_i^{(n)}\}_{i=1}^d$
\begin{equation*}
    \hat{v}_1^{(n)} = c_1v_1^{(n)} + c_2v_2^{(n)} + \dots + c_dv_d^{(n)},~~c_i={v_i^{(n)}}^\top\hat{v}_1^{(n)},~~\sum_{i=1}^dc_i^2=1.
\end{equation*}

Then we obtain
\begin{equation*}
    \begin{aligned}
        2\appv{1}\appvt{1}\nabla^2E(x^{(n)}) & = 
        2\bigg(\sum_{i=1}^d c_i v_i^{(n)} \bigg) \bigg(\sum_{i=1}^d c_i v_i^{(n)} \bigg)^\top \bigg(\sum_{j=1}^d \lambda_j^{(n)} v_j^{(n)}{v_j^{(n)}}^\top\bigg)\\
        & =  2\bigg(\sum_{i=1}^d c_i v_i^{(n)} \bigg) \bigg(\sum_{j=1}^d \lambda_j^{(n)}c_j {v_j^{(n)}}^\top\bigg)\\
        & = 2\sum_{1\leq i,j\leq d} \lambda_j^{(n)}c_jc_i v_i^{(n)} {v_j^{(n)}}^\top,
    \end{aligned}
\end{equation*}
and consequently
\begin{equation}
    \begin{aligned}
        Q^{(n)} & = I - \beta_n \bigg(\sum_{l=1}^d \lambda_l^{(n)} v_l^{(n)}{v_l^{(n)}}^\top - 2\sum_{1\leq i,j\leq d} \lambda_j^{(n)}c_jc_i v_i^{(n)} {v_j^{(n)}}^\top \bigg)\\ 
        & = \bigg[I - \beta_n \bigg(\lambda_1^{(n)}(1-2c_1^2)v_1^{(n)}{v_1^{(n)}}^\top +  \sum_{l=2}^d \lambda_l^{(n)} v_l^{(n)}{v_l^{(n)}}^\top \bigg)\bigg] \\
        & \quad + 2\beta_n\sum_{(i,j)\in S}\lambda_j^{(n)}c_jc_i v_i^{(n)} {v_j^{(n)}}^\top=:G^{(n)}+D^{(n)}.
    \end{aligned}
\end{equation}

Since $r_n<\delta$, $|\lambda_i|\leq L$ for $1\leq i\leq d$ and we apply Lemma \ref{inexact_d} to obtain
\begin{equation*}
    \|D^{(n)}\|_2\leq 2\beta_nL\|C_{-1}\|_2^2 + 2\beta_nL|c_1|\|C_{-1}\|_2.
\end{equation*}
By (\ref{cos}) we have
$
    1\geq c_1^2 = |\vt{1}\appv{1}|^2\geq 1-\alpha$ and consequently $\|C_{-1}\|^2_2=\sum_{i=2}^dc_i^2\leq \alpha$, which implies
\begin{equation}
    \|D^{(n)}\|_2\leq 2\beta_nL\alpha + 2\beta_nL\sqrt{\alpha}.
    \label{inexact_1_d}
\end{equation}
To bound $G^{(n)}$, we first note that $-\lambda_1^{(n)},\lambda_2^{(n)},...,\lambda_d^{(n)}$ belong to $[\mu,L]$, $1-2\alpha>0$ and $1-2c_1^2<0$ since $x^{(n)}\in U(x^*,\delta)$ and $c_1^2\geq1-\alpha>\frac{1}{2}$. Hence $\lambda_1^{(n)}(1-2c_1^2),\lambda_2^{(n)},...,\lambda_d^{(n)}$ belong to $[(1-2\alpha)\mu, L]$. By Lemma \ref{Q} we get
\begin{equation}
    \|G^{(n)}\|_2 \leq \frac{L-(1-2\alpha)\mu}{L+(1-2\alpha)\mu},
    \label{inexact_1_g}
\end{equation}
which, together with (\ref{inexact_1_d}) and $\beta_n = \frac{2}{L+(1-2\alpha)\mu}$, yields
\begin{equation}
    \|Q^{(n)}\|_2\leq \|G^{(n)}\|_2 + \|D^{(n)}\|_2\leq 
    \frac{L-(1-2\alpha)\mu}{L+(1-2\alpha)\mu} +  \frac{4L(\alpha+\sqrt{\alpha})}{L+(1-2\alpha)\mu}.
\end{equation}
We invoke the bounds of $||Q^{(n)}||_2$ and $||B^{(n)}||_2$ to get
\begin{equation*}
\begin{aligned}
      r_{n+1}&\leq \|Q^{(n)}\|_2 r_n + \|B^{(n)}\|_2 r_n \\
            &\leq \frac{L-(1-2\alpha)\mu+4L(\alpha+\sqrt{\alpha})}{L+(1-2\alpha)\mu}r_n + \frac{Mr_n^2}{L+(1-2\alpha)\mu}.
\end{aligned}
\end{equation*}
Define $q=q(\alpha)$ and $c=c(\alpha)$ by
\begin{equation*}
    1-q(\alpha) = \frac{L-(1-2\alpha)\mu+4L(\alpha+\sqrt{\alpha})}{L+(1-2\alpha)\mu}, ~~c(\alpha)=\frac{M}{L+(1-2\alpha)\mu}>0.
\end{equation*}
By assumptions of this theorem we have
\begin{equation*}
  0<  q(\alpha)\leq \frac{2(1-2\alpha)}{\kappa+(1-2\alpha)} <1,
\end{equation*}
which completes the proof.
\end{proof}

\begin{theorem}
   Suppose the Assumption \ref{original_asm} and (\ref{cos}) hold, 
     $\alpha$ satisfies 
    \begin{equation}
        0<\alpha<\frac{1}{2} \quad\text{and}\quad 1-2\alpha > 2\kappa(\alpha+2\sqrt{\alpha}), 
        \label{alpha_1}
    \end{equation}
    the initial point $x^{(0)}$ satisfies
    \[r_0=\|x^{(0)}-x^*\|_2<\min\{\delta, \hat{r}\},\ \hat{r}=\frac{2\mu\eta}{M}, \]
    where $\eta= 1-2\alpha-2\kappa(\alpha+2\sqrt{\alpha})>0$ with $\ds\kappa=\frac{L}{\mu}$, and $\ds\beta_n =  \frac{2}{L+(1-2\alpha)\mu}$. Then $x^{(n)}$
    converges to $x^*$ as $n\rightarrow \infty$ with the estimate on the convergence rate
   \begin{equation}\label{mh2} r_n=\|x^{(n)}-x^*\|_2\leq \left(1-\frac{2\eta}{\kappa+1-2\alpha+2\eta}\right)^n\frac{\hat{r}r_0}{\hat{r}-r_0}.\end{equation}
\end{theorem}
The proof of this theorem could be performed in parallel with that for Theorem \ref{thmorg} and is thus omitted. In particular, the decay rate in (\ref{mh2}) equals to that in Theorem \ref{thmorg} if $\alpha=0$ (i.e., the approximate eigenvector $\hat v_1^{(n)}$ equals to the exact eigenvector $v_1^{(n)}$), which demonstrates the consistency of the results.

\section{Convergence rates of index-$k$ saddle dynamics}
In this section, we analyze the convergence rates of the iterations of index-$k$ saddle dynamics. Similar to Section \ref{sec4}, both the exact orthonormal eigenvectors $\{\vv{i}\}_{i=1}^k$ of the first $k$ smallest eigenvalues of $\nabla^2 E(x^{(n)})$ and their approximations $\{\appv{i}\}_{i=1}^k$ computed via the schemes in Section \ref{sec22} will be applied in each iteration. Compared with the simplest case $k=1$ studied in Section \ref{sec4}, which has only one direction vector in the system, more complicated analysis is required in this section due to the involvement of $k$ direction vectors.

\subsection{The case of exact eigenvectors}
In this case, the iteration scheme for the position variable reads
\begin{equation}
    x^{(n+1)} =  x^{(n)}-\beta_n\bigg(I-2\sum_{i=1}^kv_i^{(n)}{v_i^{(n)}}^\top\bigg)\nabla E(x^{(n)}).
    \label{exact_position_k}
\end{equation}

\begin{theorem}
Under Assumption \ref{original_asm}, if 
$r_n:=\|x^{(n)}-x^*\|_2<\delta$ and $\ds\beta_n = \frac{2}{L+\mu}$ for some $n\geq 0$, the following estimate holds
\begin{equation*}
    r_{n+1}\leq \left(1-\frac{2\mu}{L+\mu}\right)r_n + \frac{Mr_n^2}{L+\mu}.
\end{equation*}
  \label{exact_k}
\end{theorem}

\begin{proof}
By (\ref{exact_position_k}), we choose $A^{(n)} = I-2\sum_{i=1}^kv_i^{(n)}\vt{i}$ and apply Lemma \ref{recursion} to get 
\begin{equation*}
    x^{(n+1)}-x^* = \big(Q^{(n)} + B^{(n)}\big) (x^{(n)}-x^*),~~\|B^{(n)}\|_2\leq \frac{1}{2}M\beta_nr_n,
\end{equation*}
which further implies $
    r_{n+1}\leq \|Q^{(n)}\|_2r_n + \|B^{(n)}\|_2r_n$ and we remain to bound $Q^{(n)}$.
 By eigenvalue decomposition (\ref{eigendecom}), $Q^{(n)}$ is represented as 
\begin{equation*}
    \begin{aligned}
          Q^{(n)} &= 
          I-\beta_n\left(I-2\sum_{i=1}^k v_i^{(n)}\vt{i}\right)\left(\sum_{i=1}^d \lambda_i^{(n)}v_i^{(n)}\vt{i}\right)\\
          &=I-\beta_n\left(\sum_{i=k+1}^{d}\lambda_i^{(n)}v_i^{(n)}\vt{i}-\sum_{i=1}^k\lambda_i^{(n)}v_i^{(n)}\vt{i}\right).
    \end{aligned}
\end{equation*}
By $r_n<\delta$ and Assumption 2.1,
\begin{equation*}
    -\lambda_i\in[\mu,L],\ \lambda_j\in[\mu,L],\ 1\leq i\leq k,\ k+1\leq j\leq d,
\end{equation*}
Then an application of Lemma \ref{Q} yields 
\begin{equation}
    \|Q^{(n)}\|_2\leq \frac{L-\mu}{L+\mu}. 
    \label{inexact_1_q}
\end{equation}
We finally invoke estimates of $B^{(n)}$ and $Q^{(n)}$ to bound $r_{n+1}$ to end the proof.
\end{proof}

Based on this theorem, we follow the proof of Theorem \ref{thmorg} to prove the following convergence result.
\begin{theorem}\label{thmng}
    Under Assumption \ref{original_asm}, if the initial point $x^{(0)}$ satisfies
    \[r_0=\|x^{(0)}-x^*\|_2<\min\{\delta,\hat{r}\},\ \hat{r}=\frac{2\mu}{M} \] 
 and $\beta_n = \frac{2}{L+\mu} $ for any $n\geq 0$, $x^{(n)}$ converges to $x^*$ as $n\rightarrow \infty$ 
    with the estimate on the convergence rate
    \[r_n=\|x^{(n)}-x^*\|_2\leq \left(1-\frac{2}{\kappa+3}\right)^n\frac{\hat{r}r_0}{\hat{r}-r_0},~~\kappa=\frac{L}{\mu}.\]
\end{theorem}

\subsection{The case of approximate eigenvectors}
In this subsection, we consider the practical case that the eigenvectors are computed by the schemes in Section \ref{sec22}. To analyze this realistic case, a proper measure for the distance between the exact subspace spanned by the column vectors of $V_k^{(n)}=[v^{(n)}_1,...,v^{(n)}_k]$ and the approximate subspace spanned by those of $\hat{V}_k^{(n)}=[\hat{v}_1^{(n)},...,\hat{v}_k^{(n)}]$ is needed. In this work we measure this distance from the perspective of the projection and we suppose there exists an $0\leq \alpha\leq 1$ such that
\begin{equation}\label{cos2}
\big\|V_k^{(n)}{{}V_k^{(n)}}^\top - \hat{V}_k^{(n)}{{}\hat{V}_k^{(n)}}^\top\big\|_2\leq \alpha.
\end{equation}
   As $V^{(n)}=[\vv{1},...,\vv{d}]$ and $\hat{V}^{(n)}=[\appv{1},...,\appv{d}]$ are orthogonal matrices, we apply Lemma \ref{subspace} to obtain
    \begin{equation}\label{cos3}
        \|V_k^{(n)}{{}V_k^{(n)}}^\top - \hat{V}_k^{(n)}{{}\hat{V}_k^{(n)}}^\top\|_2 = \|{{}V_{-k}^{(n)}}^\top\hat{V}_k^{(n)}\|_2\leq \alpha
    \end{equation}
   where $V_{-k}^{(n)} = [v_{k+1}^{(n)},...,v_d^{(n)}]$, and the iteration scheme of the position variable reads
\begin{equation}
\begin{aligned}
        x^{(n+1)} &=  x^{(n)}-\beta_n\bigg(I-2\sum_{i=1}^k\appv{i}\appvt{i}\bigg)\nabla E(x^{(n)})\\
        & = x^{(n)}-\beta_n\big(I-2\hat{V}_k^{(n)}{{}\hat{V}_k^{(n)}}^\top\big)\nabla E(x^{(n)}).
\end{aligned}
    \label{inexact_position_k}
\end{equation}

\begin{lemma}
\label{inexact_aux}
Under the assumption (\ref{cos2}), $\hat{V}_k^{(n)}$ {can} be represented as 
\begin{equation*}
    \hat{V}_k^{(n)} = [V_k^{(n)}, V_{-k}^{(n)}]\begin{bmatrix}C^{(n)}_k\\C^{(n)}_{-k}
    \end{bmatrix},
\end{equation*}
where $C^{(n)}_k = {{}V_k^{(n)}}^\top\hat{V}_k^{(n)}$, $C^{(n)}_{-k} = {{}V_{-k}^{(n)}}^\top\hat{V}_k^{(n)}$ and the following estimates hold
\begin{equation*}
    (1-\alpha)I\preceq C^{(n)}_k{{}C^{(n)}_k}^\top \preceq I,\ \|C^{(n)}_{-k}\|_2\leq \alpha.
\end{equation*}
\end{lemma}
\begin{proof}
Direct calculations show that 
\begin{equation*}
    \begin{aligned}
    \begin{bmatrix}
    V_k^{(n)}, V_{-k}^{(n)}
    \end{bmatrix}
    \begin{bmatrix}C^{(n)}_k\\C^{(n)}_{-k}
    \end{bmatrix} 
    &= V_k^{(n)}C^{(n)}_k + V_{-k}^{(n)}C^{(n)}_{-k}\\
    &\hspace{-0.5in}= V_k^{(n)}{{}V_k^{(n)}}^\top\hat{V}_k^{(n)} + V_{-k}^{(n)}{{}V_{-k}^{(n)}}^\top\hat{V}_k^{(n)} = V^{(n)}{{}V^{(n)}}^\top \hat{V}_k^{(n)} = \hat{V}_k^{(n)},
    \end{aligned}
\end{equation*}
where we used $V^{(n)}=[V_k^{(n)}, V_{-k}^{(n)}]$. By (\ref{cos3}) we have $ ||C^{(n)}_{-k}||_2\leq \alpha$. Besides, (\ref{cos2}) implies
\begin{equation*}
    V_k^{(n)}{{}V_k^{(n)}}^\top-\alpha I\preceq \hat{V}_k^{(n)}{{}\hat{V}_k^{(n)}}^\top \preceq V_k^{(n)}{{}V_k^{(n)}}^\top+\alpha I,
\end{equation*}
which leads to $
    {{}V_k^{(n)}}^\top[V_k^{(n)}{{}V_k^{(n)}}^\top-\alpha I]V_k^{(n)} \preceq {{}V_k^{(n)}}^\top [\hat{V}_k^{(n)}{{}\hat{V}_k^{(n)}}^\top] V_k^{(n)},
$
i.e., 
$$
    (1-\alpha)I\preceq C^{(n)}_k{{}C^{(n)}_k}^\top.$$
By $\|\hat{V}^{(n)}_k\|_2 = \|V^{(n)}_k\|_2=1$ we have
\begin{equation*}
    \|C^{(n)}_k{{}C^{(n)}_k}^\top\|_2\leq 
    \|C^{(n)}_k\|_2^2\leq \|V^{(n)}_k\|_2^2 \|\hat{V}^{(n)}_k\|_2^2 =1.
\end{equation*}
 Thus we obtain $ C^{(n)}_k{{}C^{(n)}_k}^\top \preceq I$, which completes the proof.
\end{proof}

\begin{theorem}
\label{accuracy}
Suppose the Assumption \ref{original_asm} and (\ref{cos2}) hold, $\alpha$ satisfies 
\begin{equation}
    1-\alpha > \kappa\alpha(\alpha+5),~~\kappa=\frac{L}{\mu},
    \label{inexact_k_alpha}
\end{equation}
  $r_n:=\|x^{(n)}-x^*\|_2<\delta$ and $\ds\beta_n = \frac{2}{L(1-\alpha^2)+\mu(1-\alpha)}$. The following estimate holds
\begin{equation*}
    r_{n+1}\leq \left[1-q(\alpha)\right]r_n+c(\alpha)r_n^2,
\end{equation*}
where
\begin{equation*}
    q(\alpha) = \frac{2(1-\alpha)-2\kappa\alpha(\alpha+5)}{\kappa(1-\alpha^2)+(1-\alpha)}\in(0,1),\ c(\alpha) = \frac{M\mu^{-1}}{\kappa(1-\alpha^2)+(1-\alpha)}>0.
\end{equation*}
  \label{inexact_k}
\end{theorem}
\begin{proof}
Based on the formulation (\ref{inexact_position_k}), we choose $A^{(n)} = I-2\hat{V}^{(n)}_k{{}\hat{V}^{(n)}_k}^\top$ and apply Lemma \ref{recursion} to get 
\begin{equation*}
    x^{(n+1)}-x^* = (Q^{(n)} + B^{(n)}) (x^{(n)}-x^*),~~\|B^{(n)}\|_2\leq \frac{1}{2}M\beta_nr_n,
\end{equation*}
which leads to
\begin{equation}
\begin{aligned}
      r_{n+1}&\leq \|Q^{(n)}\|_2r_n + \|B^{(n)}\|_2r_n\leq\|Q^{(n)}\|_2r_n + \frac{1}{2}M\beta_nr_n^2.
    \label{relation}
\end{aligned}
\end{equation}
To bound $\|Q^{(n)}\|_2$, we reformulate the eigenvalue decomposition (\ref{eigendecom}) as
\begin{equation*}
    \nabla^2 E(x^{(n)}) =  V_k^{(n)}\Lambda_k{{}V_k^{(n)}}^\top + V_{-k}^{(n)}\Lambda_{-k}{{}V_{-k}^{(n)}}^\top,
\end{equation*}
where $\Lambda_k^{(n)}=\text{diag}\{\lambda_1^{(n)},...,\lambda^{(n)}_k\}$ and $\Lambda_{-k}^{(n)}=\text{diag}\{\lambda_{k+1}^{(n)},...,\lambda^{(n)}_d\}$. Then by this and Lemma \ref{inexact_aux}  
\begin{equation}
    \begin{aligned}
    \hat{V}_k^{(n)}{{}{\hat{V}_k^{(n)}}}^\top\nabla^2E(x^{(n)}) = V^{(n)}_kC_k^{(n)}{{}C_k^{(n)}}^\top \Lambda_k^{(n)}{{}V^{(n)}_k}^\top + V^{(n)}_{-k}C_{-k}^{(n)}{{}C_{-k}^{(n)}}^\top \Lambda_{-k}^{(n)}{{}V^{(n)}_{-k}}^\top\\
    \ +V^{(n)}_kC_k^{(n)}{{}C_{-k}^{(n)}}^\top \Lambda_{-k}^{(n)}{{}V^{(n)}_{-k}}^\top+
    V^{(n)}_{-k}C_{-k}^{(n)}{{}C_k^{(n)}}^\top \Lambda_k^{(n)}{{}V^{(n)}_k}^\top.
    \end{aligned}
    \label{inexact_aux_1}
\end{equation}
where we used $\hat{V}_k^{(n)} = V_k^{(n)}C_k^{(n)} + V_{-k}^{(n)}C_{-k}^{(n)}$. Based on (\ref{inexact_aux_1}) we have 
\begin{equation*}
\begin{aligned}
 &\big( I-2\hat{V}_k^{(n)}{{}{\hat{V}_k^{(n)}}}^\top \big) \nabla^2E(x^{(n)})\\
  &\qquad =V^{(n)}_k\left(I-2C_k^{(n)}{{}C_k^{(n)}}^\top\right) \Lambda_k^{(n)}{{}V^{(n)}_k}^\top+ V^{(n)}_{-k}\left(I-2C_{-k}^{(n)}{{}C_{-k}^{(n)}}^\top\right) \Lambda_{-k}^{(n)}{{}V^{(n)}_{-k}}^\top\\
 &\qquad-2V^{(n)}_kC_k^{(n)}{{}C_{-k}^{(n)}}^\top \Lambda_{-k}^{(n)}{{}V^{(n)}_{-k}}^\top-2V^{(n)}_{-k}C_{-k}^{(n)}{{}C_k^{(n)}}^\top \Lambda_k^{(n)}{{}V^{(n)}_k}^\top.
\end{aligned}
\end{equation*}
If we introduce the following splittings 
$$\begin{array}{c}
\ds C_k^{(n)}{{}C_k^{(n)}}^\top=\bigg(1-\frac{\alpha}{2}\bigg)I+\bigg[C_k^{(n)}{{}C_k^{(n)}}^\top-\bigg(1-\frac{\alpha}{2}\bigg)I\bigg],\\[0.15in]
\ds C_{-k}^{(n)}{{}C_{-k}^{(n)}}^\top=\frac{\alpha^2}{2}I+\bigg(C_{-k}^{(n)}{{}C_{-k}^{(n)}}^\top-\frac{\alpha^2}{2}I\bigg),
\end{array} 
$$
then we could reformulate $( I-2\hat{V}_k^{(n)}{{}{\hat{V}_k^{(n)}}}^\top ) \nabla^2E(x^{(n)})$ as 
\begin{equation*}
  ( I-2\hat{V}_k^{(n)}{{}{\hat{V}_k^{(n)}}}^\top ) \nabla^2E(x^{(n)}) = K^{(n)} + R^{(n)},
\end{equation*}
where
\begin{equation*}
    \begin{aligned}
    K^{(n)} &= \left[1-2(1-\frac{\alpha}{2})\right]V^{(n)}_k\Lambda_k^{(n)}{{}V^{(n)}_k}^\top + (1-2\times \frac{\alpha^2}{2})V^{(n)}_{-k}\Lambda_{-k}^{(n)}{{}V^{(n)}_{-k}}^\top\\
    &=\left(\alpha-1\right)V^{(n)}_k\Lambda_k^{(n)}{{}V^{(n)}_k}^\top + (1-\alpha^2)V^{(n)}_{-k}\Lambda_{-k}^{(n)}{{}V^{(n)}_{-k}}^\top,
    \end{aligned}
\end{equation*}
and
\begin{equation*}
\begin{aligned}
R^{(n)} & = -2V^{(n)}_kC_k^{(n)}{{}C_{-k}^{(n)}}^\top \Lambda_{-k}^{(n)}{{}V^{(n)}_{-k}}^\top-2V^{(n)}_{-k}C_{-k}^{(n)}{{}C_k^{(n)}}^\top \Lambda_k^{(n)}{{}V^{(n)}_k}^\top\\ 
 &\qquad+ 2V^{(n)}_k\left[(1-\frac{\alpha}{2})I-C_k^{(n)}{{}C_k^{(n)}}^\top\right] \Lambda_k^{(n)}{{}V^{(n)}_k}^\top\\
 &\qquad+ 2V^{(n)}_{-k}\left[\frac{\alpha^2}{2}I-C_{-k}^{(n)}{{}C_{-k}^{(n)}}^\top\right] \Lambda_{-k}^{(n)}{{}V^{(n)}_{-k}}^\top\\
\end{aligned}
\end{equation*}
Thus we obtain
\begin{equation*}
    \|Q^{(n)}\|_2\leq \|I-\beta_nK^{(n)}\|_2 + \beta_n\|R^{(n)}\|_2.
\end{equation*}
It is clear that $z_i^{(n)}:=(\alpha-1)\lambda_i^{(n)}$ for $1\leq i\leq k$ and $z_j^{(n)}:=(1-\alpha^2)\lambda_j^{(n)}$ for $k+1\leq j\leq d$ are eigenvalues of $K^{(n)}$, which is a symmetric positive definite matrix by the assumptions on $\alpha$ and $x^{(n)}\in U(x^*,\delta)$. Then 
\begin{equation*}
    (1-\alpha)\mu \leq z_i \leq (1-\alpha^2)L,\ 1\leq i\leq d.
\end{equation*}
Apply Lemma \ref{Q} to get
\begin{equation}
\|I-\beta_nK\|_2\leq \frac{(1-\alpha^2)L-(1-\alpha)\mu}{(1-\alpha^2)L+(1-\alpha)\mu}\ \text{under}\ \beta_n=\frac{2}{(1-\alpha^2)L+(1-\alpha)\mu}.
\label{inexact_part_1}
\end{equation}
On the other hand we apply $\|\Lambda_{k}^{(n)}\|_2,\|\Lambda_{-k}^{(n)}\|_2\leq L$  to bound $R^{(n)}$ by
\begin{equation*}
    \begin{aligned}
    \|R^{(n)}\|_2&\leq 2\| C_k^{(n)}{{}C_{-k}^{(n)}}^\top \Lambda_{-k}^{(n)}\|_2 + 2\|C_{-k}^{(n)}{{}C_k^{(n)}}^\top \Lambda_k^{(n)}\|_2 \\
    &~+ 2\|\left[(1-\frac{\alpha}{2})I-C_k^{(n)}{{}C_k^{(n)}}^\top\right] \Lambda_k^{(n)}\|_2+ 2\|\left[\frac{\alpha^2}{2}I-C_{-k}^{(n)}{{}C_{-k}^{(n)}}^\top\right] \Lambda_{-k}^{(n)}\|_2\\
    &\leq 4L||C_k^{(n)}{{}C_{-k}^{(n)}}^\top||_2 + 2L\|(1-\frac{\alpha}{2})I-C_k^{(n)}{{}C_k^{(n)}}^\top\|_2 \\
    &+ 2L\|\frac{\alpha^2}{2}I-C_{-k}^{(n)}{{}C_{-k}^{(n)}}^\top\|_2.
    \end{aligned}
\end{equation*}
By Lemma \ref{inexact_aux}, $\|C_k^{(n)}\|_2^2 = \|C_k^{(n)}{{}C_k^{(n)}}^\top\|_2\leq 1$, $\|C^{(n)}_{-k}\|_2\leq \alpha$ and 
\begin{equation*}
\begin{array}{c}
\ds -\frac{\alpha}{2}I\preceq\left[ (1-\frac{\alpha}{2})I-C_k^{(n)}{{}C_k^{(n)}}^\top \right]\preceq \frac{\alpha}{2}I,\\[0.2in]
\ds    -\frac{\alpha^2}{2}I\preceq\left[ \frac{\alpha^2}{2}I-C_{-k}^{(n)}{{}C_{-k}^{(n)}}^\top \right]\preceq \frac{\alpha^2}{2}I.
\end{array}
\end{equation*}
Hence,
\begin{equation}
    \|R^{(n)}\|_2\leq 4L\alpha + L\alpha + L\alpha^2 = L(\alpha^2+5\alpha).
    \label{inexact_part_2}
\end{equation}
Combining (\ref{inexact_part_1}) and (\ref{inexact_part_2}) as well as the assumption $1-\alpha>\kappa\alpha(\alpha+5)$ we obtain
\begin{equation}
    \|Q^{(n)}\| _2\leq \frac{\kappa(1-\alpha^2)-(1-\alpha)+2\kappa\alpha(\alpha+5)}{\kappa(1-\alpha^2)+(1-\alpha)}<1,~~\kappa=\frac{L}{\mu}.
    \label{Q_inexact}
\end{equation}
We invoke this in (\ref{relation}) to complete the proof.
\end{proof}

Based on this theorem, we follow the proof of Theorem \ref{thmorg} to prove the following convergence result.
\begin{theorem}
Suppose the Assumption \ref{original_asm} and \ref{cos2} hold, $\alpha$ satisfies 
\begin{equation}
    1-\alpha > \kappa\alpha(\alpha+5),~~\kappa=\frac{L}{\mu},
\end{equation}
the initial point $x^{(0)}$ satisfies
    \[r_0=\|x^{(0)}-x^*\|_2<\min\{\delta, \hat{r}\},\ \hat{r}=\frac{2\mu\eta}{M}, \]
    where $\eta= 1-\alpha-\kappa\alpha(\alpha+5)>0$, and $\beta_n = \frac{2}{L(1-\alpha^2)+\mu(1-\alpha)}$. Then $x^{(n)}$
    converges to $x^*$ as $n\rightarrow \infty$ with the estimate on the convergence rate
    \[r_n=\|x^{(n)}-x^*\|_2\leq \left(1-\frac{2\eta}{\kappa(1-\alpha^2)+1-\alpha+2\eta}\right)^n\frac{\hat{r}r_0}{\hat{r}-r_0}.\] 
In particular, if $\alpha=0$, the convergence rate in the above equation equals to that in Theorem \ref{thmng}, which again demonstrates the consistency. 
\end{theorem}

\section{Numerical Experiments}
In this section, we measure the convergence rates in numerical examples to substantiate the theoretical results and to explore how the local curvature in the neighborhood of the target saddle point and the accuracy of computing eigenvectors affect the convergence rates as indicated in conclusions of previous theorems.\par 

{\bf Example 1: Impacts of local curvature.} In this example we investigate how the local curvature in the neighborhood of the target saddle point affects the convergence rates. We select the modified Powell singular function for computation, which is a four-dimensional function defined by
\begin{equation*}
    P_k(x) = P(x)  - \sum_{i=1}^k s_i\arctan^2(x_i-x^*_i) + \sum_{j=k+1}^4 s_j\arctan^2(x_j-x^*_j)
\end{equation*}
where $P(x)$ is the Powell singular function
\begin{equation*}
    P(x) = (x_1+10x_2)^2 + 5(x_3-x_4)^2 + (x_2-2x_3)^4 + 10(x_1-x_4)^4.
\end{equation*}
Here $x_i$ is the $i$th coordinate of the vector $x:=[x_1,\cdots,x_4]^\top$ and $x^*:=[0,0,0,0]^\top$. By choosing suitable coefficients $s=[s_1,s_2,s_3,s_4]^\top$, $x^*$ could be a saddle point of $P_k(x)$. We fix $k=3$ in our experiments and set $x^{(0)}=[-0.15,0.2,0,-0.2]^\top$ as the initial point. The local curvature varies by choosing different parameter vectors $s$ in the Table \ref{powell}, in which we fix $s_1, s_4$ and set $s_2=s_3$, and {\color{blue}$x^{*}$} becomes an index-2 saddle point in this case. As $s_2,s_3$ decrease, the condition number $\kappa^*$ of the Hessian $\nabla^2 P_3(x^*)$ increases, which implies that the parameter $\kappa$ in previous theorems increases and thus the convergence rate decreases. We set the step size as $\beta=0.009$ and apply simultaneous Rayleigh-quotient minimization method to calculate eigenvectors under all cases in Table \ref{powell}. We compare the convergence rates of the algorithm in Figure \ref{fig_powell_cond}.1, which exhibits that a smaller condition number implies a faster convergence rate. This observation coincides with our theoretical analysis. 
\begin{table}[h]
\caption{Coefficients $s$ and corresponding condition numbers $\kappa^*$ of the Hessian $\nabla^2 P_3(x^*)$.}
\centering
\begin{tabular}{c|ccccc}
\hline
Case & $s_1$ & $s_2$ & $s_3$ & $s_4$ & $\kappa^*$ \\ \hline
(1)   & 10    & 12    & 12    & 1     & 11.56    \\
(2)   & 10    & 6     & 6     & 1     & 26.35    \\
(3)   & 10    & 4     & 4     & 1     & 46.38    \\ \hline
\end{tabular}
\label{powell}
\end{table}

\begin{figure}[h]
    \centering
    \includegraphics[height=6cm,width=7cm]{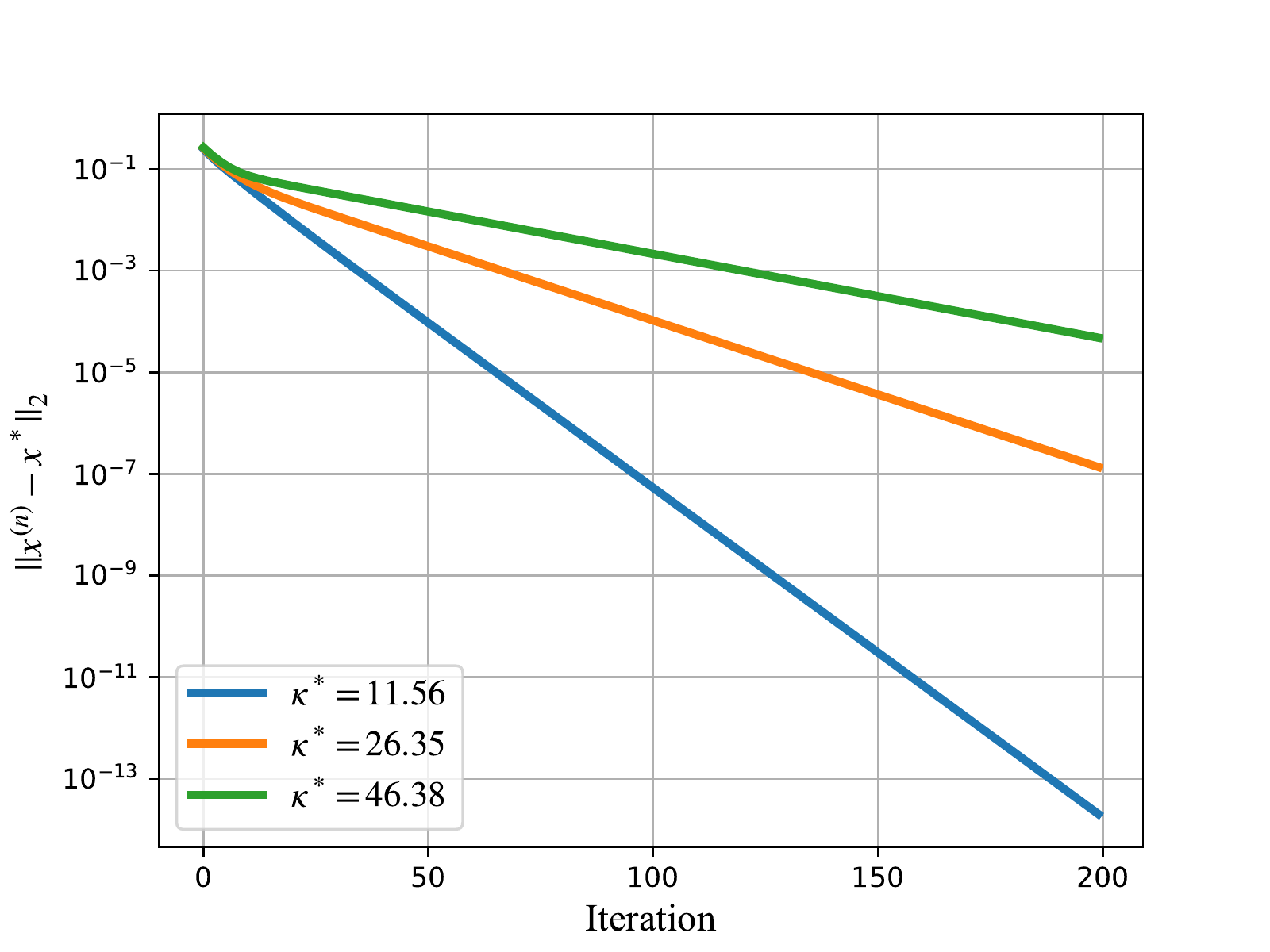}
    \label{fig_powell_cond}
    \caption{Plots of $||x^{(n)}-x^*||_2$ with respect to the iteration number under different $\kappa^*$.}
\end{figure}
\par

{\bf Example 2: Impacts of accuracy of eigenvectors.} In this case, we explore the impacts of the accuracy of the eigenvector computation on the convergence rates. Consider the following modified six-dimensional Biggs EXP6 function \cite{YinSCM}
\begin{equation*}
    B_k(x) = B(x) - \sum_{i=1}^k s_i\arctan^2(x_i-x^*_i) + \sum_{j=k+1}^6 s_j\arctan^2(x_j-x^*_j),
\end{equation*}
where $B(x)$ is the Biggs EXP6 function \cite{more1981testing}
\begin{equation*}
    B(x) = \sum_{i=1}^6 [ x_3\exp(t_ix_1) - x_4\exp(-t_ix_2) + x_6\exp(-t_ix_5) - y_i]^2.
\end{equation*}
Here $t_i = \frac{i}{10}$, $y_i = \exp(-t_i)-6\exp(-10t_i)+3\exp(-4t_i)$ and $x^* = [1,10,1,5,4,3]^\top$. Under suitable coefficients $s=[s_1,s_2,s_3,s_4,s_5,s_6]^\top$, $x^*$ becomes an index-$k$ saddle point of $B_k(x)$. In our experiments, we set $k=4$ , $s=[4,8,16,8,4,2]^\top$, the initial point $x^{(0)}=[0,9,1,5,4,3]^\top$ and the step size $\beta=10^{-4}$. In this case, $x^*$ is an index-4 saddle point of $B_k(x)$. \par 

 In numerical computations, we implement the simultaneous Rayleigh-quotient iteration minimization methods (SIRQIT) \cite{longsine1980simultaneous} and LOBPCG method \cite{knyazev2001toward} to compute eigenvectors. The sub-iteration number for  eigenvector computation is chosen as 1 or 5 for each method. In general, larger iteration number leads to more accurate computation results and LOBPCG is more accurate than SIRQIT in practice, which will accelerate the convergence according to the previous theorems. We compare both methods in Figure \ref{fig_biggs} by measuring the convergence rates, which shows that the SIRQIT with 1 sub-iteration corresponds to the slowest convergence rate since it is less accurate than other methods. By increasing the number of sub-iteration in SIRQIT, the convergence under SIRQIT becomes faster but is still a bit slower than the LOBPCG methods. The LOBPCG-based algorithm converges faster though the number of sub-iteration in the LOBPCG does not have salient impact on the convergence rate. These observations are consistent with our analysis in previous theorems.

\begin{figure}[h]
    \centering
    \includegraphics[height=6cm,width=7cm]{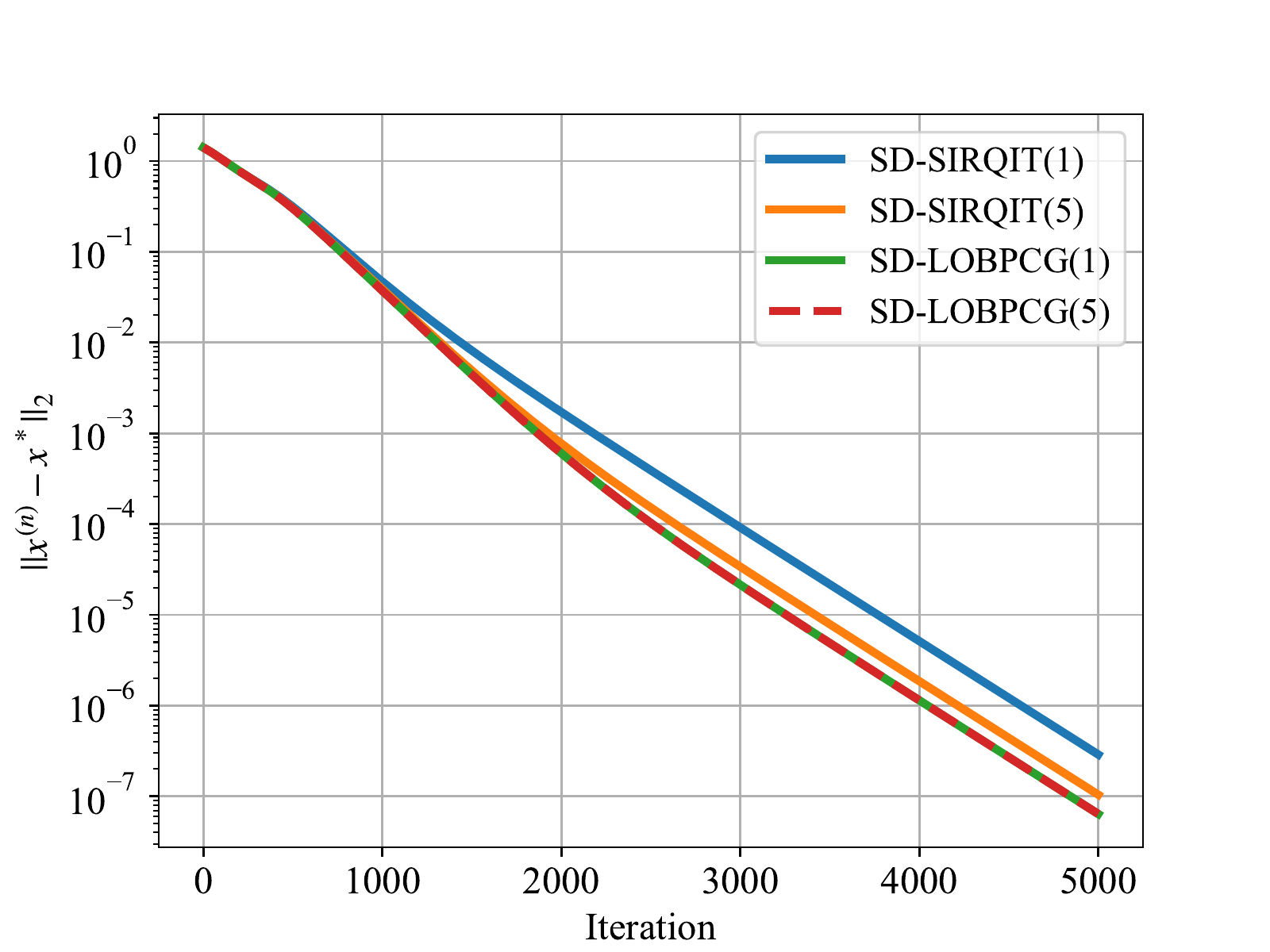}
    \caption{Plots of $||x^{(n)}-x^*||_2$ with respect to the iteration number, in which SD-SIRQIT($k$) and SD-LOBPCG($k$) represent the saddle dynamics method with SIRQIT and LOBPCG, respectively, under $k$ sub-iterations.}
    \label{fig_biggs}
\end{figure}
{
{\bf Example 3: A high-dimensional illustration.}
In this case, we consider the following high-dimensional modified Rosenbrock function
	\begin{equation*}
		B_h(x) = B(x) - \sum_{i=1}^h s_i\arctan^2(x_i-x^*_i) + \sum_{j=h+1}^d s_j\arctan^2(x_j-x^*_j),
	\end{equation*}
	where $B(x)$ is the $d-$dimensional Rosenbrock function 
	\begin{equation*}
		B(x) = \sum_{i=1}^{d-1} [100(x_{i+1} - x_i^2)^2 + (1-x_i)^2].
	\end{equation*}
 We set the dimension $d=400$, $h=20$, $s_i=200$ for $i=1,2,..,20$, $s_j=1$ for $j=21,22,...,d$, $x^{(0)}=[1.05,0.95,1.05,0.95,1,1,...,1]^\top$ and the step size $\beta=2\times10^{-4}$. In this case, $x^*:= [1,1,...,1]^\top$ is an index-5 saddle point of $B_h(x)$. \par 
	
	In numerical computations, we implement the simultaneous Rayleigh-quotient iteration (SIRQIT) method \cite{longsine1980simultaneous} and LOBPCG method \cite{knyazev2001toward} to compute eigenvectors. The sub-iteration number for eigenvector computation is chosen as 1 or 5 for each method. The observations from Figure \ref{fig_rosen} are almost the same as those in Example 2.
	
	\begin{figure}[h]
		\centering
		\includegraphics[height=6cm,width=7cm]{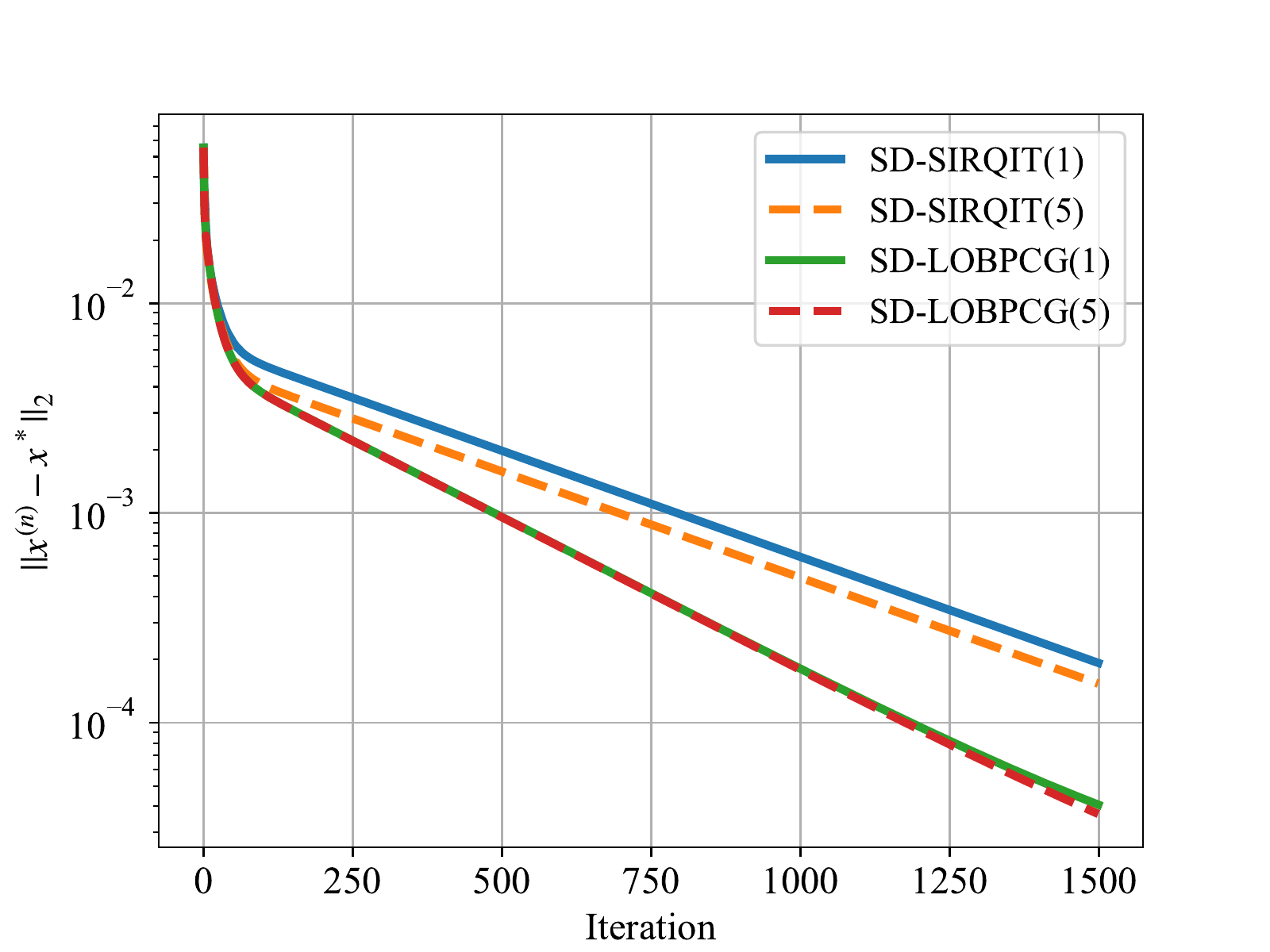}
		\caption{{Plots of $||x^{(n)}-x^*||_2$ with respect to the iteration number, in which SD-SIRQIT($k$) and SD-LOBPCG($k$) represent the saddle dynamics method with SIRQIT and LOBPCG, respectively, under $k$ sub-iterations.}}
		\label{fig_rosen}
	\end{figure}
	}

\section{Concluding remarks}
In this paper we develop systematical and novel analysis techniques to prove the local linear convergence rates of the discrete saddle dynamics, which is inspired by the local analysis of gradient descent method in the optimization theory. Our theoretical findings point out that the local curvature of the saddle point and the accuracy of the eigenvector computation are main factors that affect the convergence rates, which provides theoretical explanations for the performance of the algorithm and compensates for the convergence theory of discrete high-index saddle dynamics.  \par 

There are several potential extensions of the current work. The developed analysis techniques could be generalized to high-index saddle dynamics for non-gradient systems \cite{YinSCM}
\begin{equation*}
    \left\{
    \begin{aligned}
    \frac{dx}{dt} & = \beta\bigg(I-2\sum_{i=1}^kv_iv_i^\top\bigg)F(x),\\
    \frac{dv_i}{dt} &= \gamma\bigg[(I-v_iv_i^\top)J(x)v_i -\sum_{j=1}^{i-1}v_jv_j^\top(J(x)+J^\top(x))v_i\bigg],\ 1\leq i \leq k,
    \end{aligned}
    \right.
\end{equation*}
 where $F(x)$ refers to the natural force and $J(x)$ is the Jacobian of $F(x)$.  To be specific, the current analysis framework could be employed to analyze the saddle dynamics for non-gradient systems by replacing the spectrum of the Hessian with the eigenvalues of the Jacobian. Similarly, one could apply the techniques in this work to study the convergence of numerical discretizations of the constrained saddle dynamics \cite{yin2020constrained}. How to design the acceleration strategies for saddle dynamics to improve the convergence rates is another interesting but challenging extension. { Furthermore, if some eigenvalue of $\nabla^2E(x^*)$ approaches 0, the lower bound $\mu$ approaches 0 and $\kappa=L/\mu$ approaches infinity. According to our analysis, the convergence rate tends to 1, which causes slow convergence and thus requires the acceleration strategies. We will continue to investigate these interesting problems in the future.} 

\section*{Acknowledgments}

This work was partially supported by the National Key R\(\&\)D Program of China 2021YFF1200500; the National Natural Science Foundation of China 12050002; the International Postdoctoral Exchange Fellowship Program (Talent-Introduction Program) No. YJ20210019; the China Postdoctoral Science Foundation Nos. 2021TQ0017 and 2021M700244.

\bibliographystyle{abbrv}
\bibliography{reference}

\end{document}